\documentclass[a4paper,12pt,reqno]{amsart}
\usepackage{graphicx,epic,color}
\usepackage[utf8]{inputenc}

\numberwithin{equation}{section}

\numberwithin{figure}{section}

\newcommand{\be}{\begin{equation}}
\newcommand{\ee}{\end{equation}}
\newcommand{\beq}{\begin{eqnarray}}
\newcommand{\eeq}{\end{eqnarray}}
\newcommand{\bea}{\begin{array}{llll}}
\newcommand{\eea}{\end{array}}
\newcommand{\bi}{\begin{itemize}}
\newcommand{\ei}{\end{itemize}}
\newcommand{\bn}{\begin{enumerate}}
\newcommand{\en}{\end{enumerate}}

\newcommand{\bl}{\begin{align}}
\newcommand{\el}{\end{align}}

\newtheorem{theoreme}{Theorem}[section]

\newtheorem{lemme}{Lemma}[section]

\newtheorem{corollaire}{Corollary}[section]
\newtheorem{definition}{Definition}[section]

\newcommand{\R}{\mathbb{R}}

\newcommand{\Z}{\mathbb{Z}}

\def\G{\Gamma}

\def\RR{\rm \hbox{I\kern-.2em\hbox{R}}}
\def\NN{\rm \hbox{I\kern-.2em\hbox{N}}}
\def\ZZ{\rm {{\rm I}\kern-.28em{\rm Z}}}
\def\KK{\rm {{\rm I}\kern-.28em{\rm K}}}
\def\EE{\rm {{\rm I}\kern-.28em{\rm E}}}

\def\<{\langle}
\def\>{\rangle}

\def\e{\varepsilon}

\def\g{\gamma}
 
\def\w{\omega}

\def\({\Bigl (}
\def\){\Bigr )}

\def\ra{\longrightarrow}

\def\ra{\rightarrow}

\begin{document}

\title{ Stable sampling  and Fourier multipliers}  

\author{Basarab Matei, Yves Meyer and Joaquim Ortega-Cerd\`a}     

\address{LAGA, Université Paris Nord, 99, avenue Jean-Baptiste Clément, 93430
Villetaneuse, France}
\address{CMLA, ENS-Cachan, 61, avenue du Président-Wilson, 94235 Cachan cedex,
France}
\address{Dpt. MAIA, Universitat  de Barcelona, Gran Via 585, 08007
Bar\-ce\-lo\-na, Spain}

\email{matei@math.univ-paris13.fr, yves.meyer@cmla.ens-cachan.fr, 
jortega@ub.edu}

\date{\today} 

\thanks{The third author is supported by the Generalitat de Catalunya (grant
2009 SGR 1303) and the Spanish Ministerio de Econom\'{\i}a y Competividad
(project MTM2011-27932-C02-01). Part of this work was done while he was
staying at the Center for Advance Studies in Oslo, and he would like to express
his gratitude to the institute for the hospitality}

\begin{abstract}
 We study the relationship between stable sampling sequences for bandlimited
functions in $L^p(\R^n)$ and the Fourier multipliers in $L^p$. In the case that
the sequence is a lattice and the spectrum is a fundamental domain for the
lattice the connection is complete. In the case of irregular sequences there is
still a partial relationship. 
\end{abstract}

\maketitle

\section{Introduction}\label{sect0}
When $\w>0,\,1<p<\infty,$  and $f\in L^p(\R)$ we write $f\in E_\w^p$ if  the
Fourier  transform  ${\hat f}$ of $f$ vanishes outside $[-\w,\w]$. If  $0<h \leq
\pi/\w$ the Shannon theorem states that any $f\in E_\w^p,$ can be completely
recovered from its samples $f(kh),\, k\in \Z$.  When $h=\pi/\w,$ the map
$S_{\w}\,:f\mapsto \sqrt{h}f(kh),\,k\in \Z,$ is an  isometry between $E_\w^2$
and $\ell^2(\Z)$.  This map $S_{\w}$ is an isomorphism between $E_\w^p$ and
$\ell^p(\Z)$ when $h=\pi/\w$ and $1<p<\infty$. This fails if $p=1$ or
$p=\infty$. 
 \smallskip

The problem addressed in this paper is to extend the Shannon theorem to
functions $f\in L^p(\R^n), n>1,$  $1< p< \infty$. We are given an  
integrable compact set  $K\subset\R^n$  and we want to know if every $f\in
L^p(\R^n)$ whose Fourier transform is supported by $K$ can be recovered from its
samples on the grid $\Gamma=\Z^n$. The general case where $\Z^n$ is replaced by
an arbitrary lattice $\Gamma$ will follow by a linear change of variables.

\section{Stable sampling and stable interpolation}
 
Let  $K\subset \R^n$ be a compact set and $p\in (1, \infty)$. Let $E^p_K$ be 
the closed subspace of $L^p(\R^n)$ consisting of all $f\in L^p(\R^n)$ whose
Fourier transform ${\hat f}$ is supported by $K$.
\begin{equation}
E^p_K= \{f \in  L^p(\R^n) : \mbox{supp} {\hat f}\subset   K\}
\end{equation}
The Fourier transform ${\hat f}$ of  $f$ is
defined by 
\begin{equation}
{\hat f}(\xi)=\int_{\R^n} \exp( - i x \cdot \xi )f(x)dx,\,\,\,
\xi\in\R^n,
\end{equation}
where $x\cdot \xi = x_1\xi_1+\cdots+x_n\xi_n$. When $1\leq p \leq 2$ the Fourier
transform ${\hat f}$ belongs to $L^q$ with $1/p+1/q=1$. If $p>2$ this Fourier
transform is defined in the distributional sense.  The following lemmata are
well known. 
 
\begin{lemme}\label{Plancherel}
If $K\subset \R^n$ is a compact set and $1\leq p \leq \infty$ the  restriction
operator $S^p$ which is defined by
\begin{equation}
f \mapsto (f(k))_{k\in \Z^n}
\end{equation}
is continuous form  $E^p_K$ to $\ell^p(\Z^n)$.
\end{lemme}
The restriction of $f$ to any uniformly separated sequence belongs to $\ell^p$.
This the well-known Plancherel-Polya inequality.
\begin{lemme}[Poisson formula]\label{Poisson}
If $F$ is a compactly supported distribution and if $G(x)=\sum_{k\in
\Z^n}F(x-2\pi k)$ then the Fourier coefficients $c(k),\,k\in \Z^n,$ of $G$ and
the Fourier transform ${\hat F}$ of $F$ are related by 
\begin{equation}
c(k)=(2\pi)^{-n}{\hat F}(k),\,k\in\Z^n.
\end{equation}
\end{lemme}

If $1 < p <\infty$ we denote by ${\mathcal F}L^p$ the Banach space of all
Fourier transforms of functions in $L^p(\R^n)$. The norm of ${\hat f}$ in
${\mathcal F}L^p$ is the norm of $f$ in $L^p$. The space of restrictions of
${\mathcal F}L^p$ to an open set  $\Omega$ will be denoted by ${\mathcal
F}L^p(\Omega)$. Similarly ${\mathcal F}\ell^p$ will denote the Banach space
consisting of all $2\pi \Z^n$ periodic functions (or distributions) whose
Fourier coefficients belong to $\ell^p(\Z^n)$ and its norm is the $\ell^p$ norm
of its coefficients.

\begin{lemme}\label{product}
For $F\in {\mathcal F}\ell^p$ and for $\phi$ in the Schwartz class  ${\mathcal
S}$ the product $\phi F$ belongs to ${\mathcal F}L^p$
\end{lemme}
Indeed let $F(x)=\sum_{k\in \Z^n}c(k)\exp(i k\cdot x)$ where $c(k)\in \ell^p$
and let $\theta$ be the Fourier transform of $\phi$. Then the Fourier transform
of $\phi F$ is $u(x)=\sum_{k\in \Z^n}c(k)\theta(x-k)$. By the Holder inequality
\[
 |u(x)|^p\leq \Bigl( \sum_{k\in \Z^n} |c(k)|^p
\theta(x-k)\Bigr)\Bigl(\sum_{k}\theta(x-k)\Bigr)^{p-1}.
\]
Thus it follows that $\|u\|_p^p\le C \sum_{k\in\Z^n} |c(k)|^p$.

 \medskip
 
From now on the compact set $K$ is assumed to be {\it regular}.
\begin{definition} 
A compact set $K\subset {\mathbb R}^n$ is {\it regular} if the following two
conditions hold 
\begin{itemize}
\item[(a)] $K$ is connected
\item[(b)] for each $x_0$ belonging to the boundary $\Gamma=\partial K$ of $K$
there exist a neighborhood $V$ of $x_0$, a suitable coordinate system ${\mathcal
R}_{x_0}$ and a continuous function $A_{x_0}\,: {\mathbb R}^{n-1}\mapsto
{\mathbb R}$ such that $\Gamma$ coincides on $V$ with  the graph of $A_{x_0}$ in
${\mathcal R}_{x_0}$ and $K$ coincides on $V$ with $\{x_n\geq A_{x_0}(x'),\,
x'=(x_1,\ldots, x_{n-1})\}.$ 
\end{itemize}
\end{definition}

A Lipschitz domain is regular. Definition 2.1 is required in this note since
the counter example which is given below is not a Lipschitz domain.   

We denote by ${\mathcal S}={\mathcal S}(\R^n)$ the Schwartz class. We then have.
\begin{lemme}~\label{densitat}
If $1<p <\infty$ and if $K$ is regular, then  ${\mathcal S}\cap E^p_K$ is dense
in $E^p_K$.
\end{lemme}
\begin{proof} Let $f\in E^p_K.$ Then the compact support
$L$ of $F={\hat f}$ is contained  in $K.$   Using a smooth partition of the
identity  we split $F$ into a finite sum of pieces for which the local
description of $K$ can be used. Then the interior $\Omega$ of $K$ is locally
defined by $x_n > A(x_1,\ldots, x_{n-1})$ while $K$ is defined by $x_n \geq
A(x_1,\ldots, x_{n-1})$ where $A$ is a continuous function. Let $F_\e=F(x_1,
\ldots, x_{n-1}, x_n-\e)$ where $\e>0$ is a small positive number. The support
$L_\e$ of $F_\e$ is $L_\e= L +\e e_n$ where $e_n=(0, \ldots, 0,1)$. Therefore
$L_\e$ is contained in $\Omega$. Moreover
$\|F_\e-F\|_{{\mathcal F}L^p}$ tends to 0 with $\e$. Indeed for every $f\in L^p$
Lebesgue dominated convergence theorem implies that  $\|[\exp(i \e x_n)-1]f\|_p
\rightarrow 0$ as $\e$ tends to 0.    To conclude the
proof it suffices to approximate $F$ by a test function whose compact support 
is contained in $\Omega.$  To reach this goal we  replace $F_\e$ by
the convolution product $F_\e\star \theta$ where  $\theta$ is a smooth bump
function supported by a sufficiently small ball $|x|\leq \eta$ with $\int
\theta=1$.   
\end{proof}

Lemma~\ref{densitat} does not hold for a Riemann integrable compact set $K$.
Here is a counter example in two dimensions. Let $K$ be  the circle centered at
$0$ with radius $1$. The compact set $K$ is Riemann integrable.   We consider
the arc length measure $d\sigma$ on $K$ and its inverse Fourier transform $f$,
that is a Bessel function that decays as $|x|^{-1/2}$ when $x\to\infty$. 
Thus $f$ belongs to $E_K^p$ for $p>4$ but $f$ is not the limit in $L^p$ of a
sequence of test functions in $E_K^p$ because there are no test functions in
$E_K^p$. Any test function $g$ will belong to $L^2(\R^2)$ and $\hat g$ will be
supported in $K$ that has measure 0, thus it will vanish.
\medskip
  
We now follow  the  seminal work of H.\,J.\,Landau \cite{Landau}.

 \begin{definition}

A point set  $\G\subset \R^n$ is a set of stable sampling for $E^p_K$ if there
exists a constant $C$ such that 
\begin{equation}\label{defsamp}
f\in E^p_K \Rightarrow \|f\|_p^p\leq C \sum_{\g\in \G}| f(\g)|^p
\end{equation}
\end{definition}

In other words $\G\subset \R^n$ is a set of stable sampling for $E^p_K$  if
$S^p: E_K^p\mapsto \ell^p(\G)$ is an isomorphism between $E_K^p$ and its image
in $\ell^p(\G)$. If this condition is not satisfied two problems may occur. The
first one is named aliasing. Aliasing means that there exists a function $f\in
E^p_K,\,f\neq0,$ such that $f(\g)=0,\,\g\in\G$. Even if aliasing does not occur
the reconstruction of $f$ from its samples $f(\g),\,\g\in\G,$ is not stable when
\eqref{defsamp} is not satisfied. Let $|E|$ denote the Lebesgue measure of a set
$E$. 

When $\Gamma=\Z^n$, by Lemma~\ref{Poisson} we can ``code'' the samples $f(\g)$
in a periodic function $G\in \mathcal F(\ell^p)$ as follows. Let $F$ be the
Fourier transform of $f$. It has compact support and we can periodize it:
\[
 G(x)=\sum_{k\in \Z^n} F(x-2\pi k).
\]
By Lemma~\ref{Poisson}, the Fourier coefficients $c(k)$ of $G$ satisfy
$c(k)=f(-k)$. Hence, the sampling inequality \eqref{defsamp}, on the Fourier
side, amounts to say that all $F\in \mathcal F(L^p)$ supported in $K$ are
controlled
by its periodized $G\in \mathcal F(\ell^p)$.

\begin{lemme}\label{open}
Let $K\subset \R^n$ be a Riemann integrable compact set. 
\begin{itemize}
\item[(a)] If $\Z^n$ is a set of stable sampling for $E^p_K$ then 
\begin{equation}
\forall k\in \Z^n,\,k\neq 0 \Rightarrow |K\cap (K+2\pi k)|=0. \label{disjoint}
\end{equation}
\item[(b)] Conversely if the sets $K+2\pi k,\,k\in \Z^n,$ are pairwise
disjoint, then $\Z^n$ is a set of stable sampling for $E^p_K$.
\end{itemize} 
\end{lemme}

Observe that the sufficient condition (b) is more demanding than (a) and
Lemma~\ref{open} does not fully answer the problem.

\smallskip

\begin{proof} If \eqref{disjoint} is not satisfied there exists a
$k_0\in \Z^n$
such that the Riemann integrable set $K\cap (K-2\pi k_0)$ has a positive
measure. Therefore this set  contains a small ball $B$. Then $B\subset K$ and
$B+2\pi k_0\subset K$. Let $f$ be any test function whose Fourier transform is
supported by $B$. Then $g(x)=(\exp(2\pi ik_0\cdot x)-1)f(x)$ belongs to $E_K^p$
and vanishes on $\Z^n$. Aliasing occurs. 

\smallskip
We now prove (b). Let $F$ be the Fourier transform of the function $f\in
E^p_K$ and $G$ its periodized version as above. If $\e>0$ is small enough and if
$B(0,\e)$ is the ball centered at 0 with radius $\e,$ the compact set
$K'=K+B(0,\e)$ still satisfies (b).  Let $\phi$ be a test function supported by
$K'$ and such that $\phi(x)=1$ on $K$. This implies $F=\phi G$ and we can apply
Lemma~\ref{product}.
\end{proof}

Here is an example illustrating Lemma~\ref{open}. Let us assume that $K$ is the
disc $|x|\leq r$. If $0<r<\pi$ condition (b) is satisfied and if $r>\pi$
condition (a) does not hold. However Lemma~\ref{open} does not give any answer
if $r=\pi$. To treat this case, we consider the cube $Q=[-\pi, \pi]^n$ to which
Theorem~\ref{theresult} below can be applied. Therefore $\Z^n$ is a set of
stable sampling for $E^p_Q$. Since $K\subset Q,$ $\Z^n$ is a set of stable
sampling for $E^p_K$.
\medskip

When $p=2,$ stable sampling is equivalent to \eqref{disjoint}.  The goal of this
note is to show that this property does not suffice when  $p\neq 2$. 

\begin{definition}

We say that  $\G\subset \R^n$ is a set of  stable interpolation 
if every  sequence $a(\g)\in\ell^p(\G)$  can be interpolated by a function $f$
in $E^p_K$. 
\end{definition}
It means that there exists $f\in E^p_K$ such that $f(\g)=a(\g),\,\g\in\G$. 
\begin{lemme}\label{openinterp}
Let $K\subset \R^n$ be a Riemann integrable compact set. 
\begin{itemize}
\item[(a)] If $\Z^n$ is a set of stable interpolation for $E^p_K$ then 
\begin{equation}
\bigcup_{k\in \Z^n}(K+2\pi k)=\R^n.
\end{equation}
\item[(b)] Conversely if $\Omega$ is the interior of $K$ and if 
\begin{equation}
\bigcup_{k\in \Z^n}(\Omega+2\pi k)=\R^n
\end{equation}
then $\Z^n$ is a set of stable interpolation for $E^p_K$.
\end{itemize} 
\end{lemme}
Here also the sufficient condition (b) is more demanding than (a). 
The proof is similar to the one used in Lemma~\ref{open}. It is left to the
reader. 

\medskip

We want to know when $ S^p: E^p_K \ra \ell^p(\Z^n)$ is an isomorphism. If it is
the case, Lemma~\ref{open} and \ref{openinterp} imply that the translated sets
$K+2\pi k,\,k\in \Z^n,$ are a partition of $\R^n$ up to sets of measure 0. We
then say that $K$ is a {\it fundamental domain} for $\Z^n$. But the converse
implication is not true.  The fact that $K$ is a {\it fundamental domain} for
$\Z^n$ does not imply that the operator $S^p$ is an isomorphism between  $E^p_K$
and $\ell^p(\Z^n)$ when $p\neq 2$. This will be proved in the next section.

\section{Our Result}

A Borel function
$m(x)$ is a multiplier of 
${\mathcal F}L^p$ if we have 
\begin{equation}
F(x)\in {\mathcal F}L^p \Rightarrow m(x)F(x)\in {\mathcal
F}L^p 
\end{equation}
and if a constant $C$ exists such that 
\begin{equation}
\|m(x)F(x)\|_{{\mathcal F}L^p}\leq C\|F(x)\|_{{\mathcal F}L^p}
\label{multiplier}
\end{equation} 
This does not make any sense if $2 < p$ since $F(x)$ may be a distribution.
This issue is settled by the following remarks. It suffices to prove
\eqref{multiplier} when
$f\in L^2\cap L^p$ and a density argument yields the general case. Moreover if
$m$ is a multiplier of ${\mathcal F}L^p$ then $m$ is also a multiplier of
${\mathcal F}L^q$ when $1/p+1/q=1$. This reduces the case $p>2$ to $1<q<2$.
\medskip

The following lemma will be seminal in the proof of Theorem~\ref{theresult}.
\begin{lemme}\label{seminal}
Let $m\in L^{\infty}(\R^n)$ be a compactly supported function. Then the
following two properties are equivalent
\begin{itemize}
\item[(a)]$m$ is a multiplier of ${\mathcal F}L^p$ 
\item[(b)] $m$ maps  ${\mathcal F}\ell^p$ into ${\mathcal F}L^p$
\end{itemize} 
\end{lemme}
\begin{proof}
For proving (a)$\Rightarrow$(b) let us denote by $\phi$ a smooth and compactly
supported function such that $\phi=1$ on a neighborhood of the compact support
of $m$.   If $G\in {\mathcal F}\ell^p$ we have $\phi G \in {\mathcal F}L^p$ by
Lemma~\ref{product} and $mG=m\phi G\in {\mathcal F}L^p$ since $m$ is a
multiplier of ${\mathcal F}L^p$. We now prove (b)$\Rightarrow$(a). If  $F\in
{\mathcal F}L^p$ we have $mF=m\phi F$ as above. One uses a smooth partition of
the identity to decompose $\phi F$ into a finite sum $F=\sum_1^m F_j$ where 
$F_j\in {\mathcal F}L^p$ and where the support of each $F_j$  is contained in a
cube $Q_j$ centered at $x_j$ with side length 1. It suffices to show that
$mF_j\in {\mathcal F}L^p$  to conclude. We consider $G_j(x)=\sum_{k\in
\Z^n}F_j(x-2\pi k)$ and Lemmata~\ref{Plancherel} and \ref{Poisson} yield $G_j\in
{\mathcal F}\ell^p$. Since $m$ maps ${\mathcal F}\ell^p$ into ${\mathcal F}L^p$
we have $mG_j\in {\mathcal F}L^p$. Let $\chi_j$ be a test function such that
$\chi_j(x)=1$ on $Q_j$ and $\chi_j(x)=0$ on every $Q_j+2k\pi,\,k\neq 0$. Then
$mF_j=\chi_j mG_j\in {\mathcal F}L^p$. 
\end{proof}

\begin{theoreme}\label{theresult}

Let $K$ be  a regular compact set. Assuming that $K$ is fundamental domain for
$2\pi\Z^n$ the following four properties are equivalent ones 
\begin{itemize}
  \item  [(a)] The indicator function of $K$ is a multiplier of ${\mathcal
F}L^p$.
   \item [(b)] The lattice $\Z^n$ is a set of stable sampling for $E^p_K$. 
   \item [(c)] The lattice $\Z^n$ is a set of stable interpolation for $E^p_K. $
   \item [(d)] The operator $S^p\,: E^p_K\mapsto \ell^p(\Z^n)$ is an
isomorphism. 
\end{itemize}
\end{theoreme}

\begin{proof}

Let us show that (a) implies  (b). For $f\in E_k^p$ we consider $F$ and $G$ as
above. If $K$ is fundamental domain for $\Z^n$, then formally $F=\chi_K G$. A
priori $G$ could be a distribution if $p>2$, so we need to assume that $f\in
\mathcal{S}\cap E_k$ which we may by Lemma~\ref{densitat}. The property of
stable sampling is first proved if $f\in {\mathcal S}\cap E^p_K$ and extended by
continuity to the general case. Hence the sampling inequality \eqref{defsamp}
amounts to
\begin{equation}\label{alternatiu}
\|\chi_K G\|_{\mathcal FL^p}\leq \|G\|_{\mathcal F \ell^p},\qquad f\in E^p_K.
\end{equation}
This immediately shows, by Lemma~\ref{seminal} that (a) implies (b). The
converse is also true because we can start from $G\in \mathcal F(\ell^p)$ and
take $F=\chi_K G$ and $f$ the function that has $F$
as its Fourier transform. The assumption
(b) means

\begin{equation}
 \|F\|_{{\mathcal F}L^p} \leq C \|f(k)\|_{\ell^p}\label{3b}
\end{equation}
which is the same as
\begin{equation} 
 \|F\|_{{\mathcal F}L^p} \leq C\|G\|_{{\mathcal F}\ell^p}. 
\end{equation}

\medskip

%
%

The equivalence between (a) and (c) follows just by looking to the situation in
terms of the sequence of Fourier coefficients of $G$, which is an arbitrary
sequence in $\ell^p$.

The equivalence between (a) and (d) obviously follows from the preceding steps.
\end{proof}
\section{An example}
\label{sect3}
\subsection{Fefferman's theorem and stable sampling}

Fefferman's theorem says that in any dimension $n\geq 2$ the indicator function
$\chi_B$ of a ball $B$ is not a multiplier of ${\mathcal F}L^p$ when $p\neq 2$. 
Moreover the proof of Fefferman's theorem or an elementary reasoning shows that
the result is local. If $\phi$ is any function which does not vanish identically
on $\partial B$ then $\phi(x)\chi_B(x)$ is not a multiplier of ${\mathcal F}L^p$
if $p\neq 2$. This paves the way to our example.

\medskip

The compact set  $K\subset \R^2$ is defined as follows. We start with   the 
square $Q=\{0 \leq  x_1 , x_2 \leq  2\pi\}$ and we call  $D_1$ (resp. $D_2$) the
closed discs centered at  $x_1 = (\pi, 0)$ (resp.   $x_2 = (\pi, 2\pi)$) with
radius $\pi$. Let $K_1 = Q \cup  D_1$ and  $K = K_1 \setminus  D_2$. It is
trivial to prove that $K$ is a fundamental domain for $2\pi \Z^2$ as in the
picture:
\bigskip
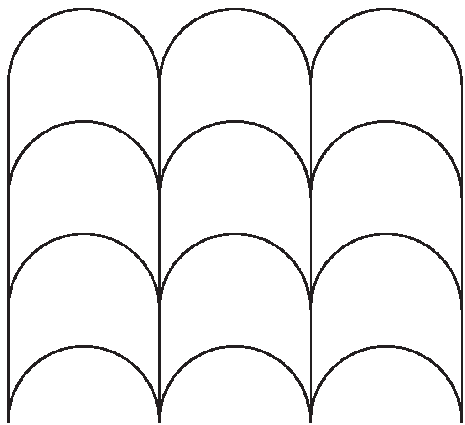
\begin{center}
\setlength{\unitlength}{0.00087489in}
\begingroup\makeatletter\ifx\SetFigFont\undefined
\def\x#1#2#3#4#5#6#7\relax{\def\x{#1#2#3#4#5#6}}%
\expandafter\x\fmtname xxxxxx\relax \def\y{splain}%
\ifx\x\y   
\gdef\SetFigFont#1#2#3{%
  \ifnum #1<17\tiny\else \ifnum #1<20\small\else
  \ifnum #1<24\normalsize\else \ifnum #1<29\large\else
  \ifnum #1<34\Large\else \ifnum #1<41\LARGE\else
     \huge\fi\fi\fi\fi\fi\fi
  \csname #3\endcsname}%
\else
\gdef\SetFigFont#1#2#3{\begingroup
  \count@#1\relax \ifnum 25<\count@\count@25\fi
  \def\x{\endgroup\@setsize\SetFigFont{#2pt}}%
  \expandafter\x
    \csname \romannumeral\the\count@ pt\expandafter\endcsname
    \csname @\romannumeral\the\count@ pt\endcsname
  \csname #3\endcsname}%
\fi
\fi\endgroup
{\renewcommand{\dashlinestretch}{30}
\begin{picture}(2724,2511)(0,-10)
\drawline(912.000,687.000)(905.863,761.068)(887.618,833.115)
	(857.763,901.176)(817.113,963.396)(766.777,1018.076)
	(708.127,1063.725)(642.763,1099.098)(572.468,1123.230)
	(499.161,1135.463)(424.839,1135.463)(351.532,1123.230)
	(281.237,1099.098)(215.873,1063.725)(157.223,1018.076)
	(106.887,963.396)(66.237,901.176)(36.382,833.115)
	(18.137,761.068)(12.000,687.000)
\drawline(912.000,1362.000)(905.863,1436.068)(887.618,1508.115)
	(857.763,1576.176)(817.113,1638.396)(766.777,1693.076)
	(708.127,1738.725)(642.763,1774.098)(572.468,1798.230)
	(499.161,1810.463)(424.839,1810.463)(351.532,1798.230)
	(281.237,1774.098)(215.873,1738.725)(157.223,1693.076)
	(106.887,1638.396)(66.237,1576.176)(36.382,1508.115)
	(18.137,1436.068)(12.000,1362.000)
\drawline(912.000,2037.000)(905.863,2111.068)(887.618,2183.115)
	(857.763,2251.176)(817.113,2313.396)(766.777,2368.076)
	(708.127,2413.725)(642.763,2449.098)(572.468,2473.230)
	(499.161,2485.463)(424.839,2485.463)(351.532,2473.230)
	(281.237,2449.098)(215.873,2413.725)(157.223,2368.076)
	(106.887,2313.396)(66.237,2251.176)(36.382,2183.115)
	(18.137,2111.068)(12.000,2037.000)
\drawline(1812.000,12.000)(1805.863,86.068)(1787.618,158.115)
	(1757.763,226.176)(1717.113,288.396)(1666.777,343.076)
	(1608.127,388.725)(1542.763,424.098)(1472.468,448.230)
	(1399.161,460.463)(1324.839,460.463)(1251.532,448.230)
	(1181.237,424.098)(1115.873,388.725)(1057.223,343.076)
	(1006.887,288.396)(966.237,226.176)(936.382,158.115)
	(918.137,86.068)(912.000,12.000)
\drawline(1812.000,687.000)(1805.863,761.068)(1787.618,833.115)
	(1757.763,901.176)(1717.113,963.396)(1666.777,1018.076)
	(1608.127,1063.725)(1542.763,1099.098)(1472.468,1123.230)
	(1399.161,1135.463)(1324.839,1135.463)(1251.532,1123.230)
	(1181.237,1099.098)(1115.873,1063.725)(1057.223,1018.076)
	(1006.887,963.396)(966.237,901.176)(936.382,833.115)
	(918.137,761.068)(912.000,687.000)
\drawline(1812.000,1362.000)(1805.863,1436.068)(1787.618,1508.115)
	(1757.763,1576.176)(1717.113,1638.396)(1666.777,1693.076)
	(1608.127,1738.725)(1542.763,1774.098)(1472.468,1798.230)
	(1399.161,1810.463)(1324.839,1810.463)(1251.532,1798.230)
	(1181.237,1774.098)(1115.873,1738.725)(1057.223,1693.076)
	(1006.887,1638.396)(966.237,1576.176)(936.382,1508.115)
	(918.137,1436.068)(912.000,1362.000)
\drawline(1812.000,2037.000)(1805.863,2111.068)(1787.618,2183.115)
	(1757.763,2251.176)(1717.113,2313.396)(1666.777,2368.076)
	(1608.127,2413.725)(1542.763,2449.098)(1472.468,2473.230)
	(1399.161,2485.463)(1324.839,2485.463)(1251.532,2473.230)
	(1181.237,2449.098)(1115.873,2413.725)(1057.223,2368.076)
	(1006.887,2313.396)(966.237,2251.176)(936.382,2183.115)
	(918.137,2111.068)(912.000,2037.000)
\drawline(2712.000,12.000)(2705.863,86.068)(2687.618,158.115)
	(2657.763,226.176)(2617.113,288.396)(2566.777,343.076)
	(2508.127,388.725)(2442.763,424.098)(2372.468,448.230)
	(2299.161,460.463)(2224.839,460.463)(2151.532,448.230)
	(2081.237,424.098)(2015.873,388.725)(1957.223,343.076)
	(1906.887,288.396)(1866.237,226.176)(1836.382,158.115)
	(1818.137,86.068)(1812.000,12.000)
\drawline(2712.000,687.000)(2705.863,761.068)(2687.618,833.115)
	(2657.763,901.176)(2617.113,963.396)(2566.777,1018.076)
	(2508.127,1063.725)(2442.763,1099.098)(2372.468,1123.230)
	(2299.161,1135.463)(2224.839,1135.463)(2151.532,1123.230)
	(2081.237,1099.098)(2015.873,1063.725)(1957.223,1018.076)
	(1906.887,963.396)(1866.237,901.176)(1836.382,833.115)
	(1818.137,761.068)(1812.000,687.000)
\drawline(2712.000,1362.000)(2705.863,1436.068)(2687.618,1508.115)
	(2657.763,1576.176)(2617.113,1638.396)(2566.777,1693.076)
	(2508.127,1738.725)(2442.763,1774.098)(2372.468,1798.230)
	(2299.161,1810.463)(2224.839,1810.463)(2151.532,1798.230)
	(2081.237,1774.098)(2015.873,1738.725)(1957.223,1693.076)
	(1906.887,1638.396)(1866.237,1576.176)(1836.382,1508.115)
	(1818.137,1436.068)(1812.000,1362.000)
\drawline(2712.000,2037.000)(2705.863,2111.068)(2687.618,2183.115)
	(2657.763,2251.176)(2617.113,2313.396)(2566.777,2368.076)
	(2508.127,2413.725)(2442.763,2449.098)(2372.468,2473.230)
	(2299.161,2485.463)(2224.839,2485.463)(2151.532,2473.230)
	(2081.237,2449.098)(2015.873,2413.725)(1957.223,2368.076)
	(1906.887,2313.396)(1866.237,2251.176)(1836.382,2183.115)
	(1818.137,2111.068)(1812.000,2037.000)
\drawline(912.000,12.000)(905.863,86.068)(887.618,158.115)
	(857.763,226.176)(817.113,288.396)(766.777,343.076)
	(708.127,388.725)(642.763,424.098)(572.468,448.230)
	(499.161,460.463)(424.839,460.463)(351.532,448.230)
	(281.237,424.098)(215.873,388.725)(157.223,343.076)
	(106.887,288.396)(66.237,226.176)(36.382,158.115)
	(18.137,86.068)(12.000,12.000)
\drawline(12,12)(12,2037)
\drawline(912,12)(912,2037)
\drawline(1812,12)(1812,2037)
\drawline(2712,12)(2712,2037)
\end{picture}
}

\end{center}

But we know
from Fefferman's theorem that the indicator function  of $K$ is not a Fourier
multiplier of ${\mathcal F}(L^p)$ if $p\neq 2$.  Therefore the functions $f\in
E_K^p$ cannot be sampled on $\Z^2$ if $p\neq 2$.

\section{Irregular sequences and multipliers}

A sequence $\Lambda\subset \R^n$ is a complete interpolating sequence for
$E_K^p$ when for any sequence of values $v_\lambda\in \ell^p(\Lambda)$
there is a \emph{unique} function $f\in E_K^p$ such that
$f(\lambda)=v_\lambda$. It can be seen that complete interpolating sequences
are simultaneously stable sampling and stable interpolating sequences.

We are not assuming now that $\Lambda$ has any structure. In such generality
very little is known. In the case of dimension one, and $K$ being an interval
there is a complete description of the complete interpolating sequences
in $E_I^p$ due to Lyubarskii and Seip \cite{LyuSei97}  that generalizes to any
$p>1$ the description of Pavlov, \cite{Pavlov} when $p=2$. 

Nevertheless it is still possible to see that the existence of complete
interpolating sequences is related to the boundedness of the multiplier at
least in one direction.

\begin{theoreme}
Given any compact $K$, and a complete interpolating sequence
$\Lambda$ for the space $E_K^p$, the function $\chi_K$ is a multiplier for
$\mathcal FL^p$
\end{theoreme}

\begin{proof}
Take a smooth compactly supported function $\phi$ such that $\phi\equiv 1$ in a
neighborhood of $K$. If $\widehat F=\phi$, then for any $f\in E_K^p$ we
have the reproducing formula: 
\begin{equation}
f(z)= f\star F (z)=\int_{\R^n} f(x)F(z-x) dx.
\end{equation}
Consider the following two operators:

\begin{equation}
\begin{split} 
T_1: L^p(\R^n)\to \ell^p(\Lambda)\\ 
f\to \{f\star F(\lambda)\}_{\lambda\in\Lambda}
\end{split}
\end{equation}

Clearly the operator $T_1$ is linear and bounded by the Plancherel-Polya
inequality since the sequence $\Lambda$ is
separated and $f\star F\in E_B^p$, where $B$ is a
ball containing the support of $\phi$.

Now consider the operator:

\begin{equation}
T_2: \ell^p(\Lambda)\to E_K^p
\end{equation}

that to any sequence of values $\{v_\lambda\}\in \ell^p(\Lambda)$ associates
the unique function $f\in E_K^p$ such that $f(\lambda)=v_\lambda$. 
This function exists because we assume that $\Lambda$ is a complete
interpolating sequence. Moreover it defines a bounded linear operator. Thus the
composition operator $T=T_2\circ T_1$ maps $L^p(\R^n)$ to $E^p_K$
linearly and it is bounded. Because of the reproducing property, and the fact
that $\Lambda$ is a uniqueness set for $E_K^p$ it follows that $T$ is a
projection, i.e. $T\circ T =T$. 

We are going to produce now another bounded projection invariant under 
translations. Denote by $\tau_x$ the translation operator by $x$. Then
$T_x=\tau_{-x} T \tau_{x}$ is another projection with the same norm as $T$. 
We average them over a big ball and denote
\begin{equation}
 T_R = \frac 1{|B(0,R)|} \int_{B(0,R)} T_x\, dx.
\end{equation}
The operator $T_R$ is again a projection with norm bounded by the norm of $T$.
Since all the operators $T_R:L^p\to L^p$ are bounded uniformly, by the
Banach-Alouglou
theorem we can extract a sequence
sequence $R_n\to \infty$ such that $T_{R_n}$ converges to
a bounded operator $\widetilde T$ in the weak operator topology, i.e $\langle
f,T_{R_n}(g)\rangle \to \langle f,\widetilde{T}(g)\rangle$ for all $f\in L^q$
and $g\in L^p$. The linear operator $\widetilde T$
is a projection and has bounded norm (all these properties are inherited from
the $T_{R_n}$). 
We will see now that it commutes with the translations.
Indeed if we fix $y\in\R^n$,
\begin{equation}
\begin{aligned}
\tau_y T_R &= \frac 1{|B(0,R)|}\int_{B(0,R)}\tau_{y-x} T \tau_x\, dx=\\
&=\frac 1{|B(0,R)|}\int_{B(y,R)}\tau_{-x} T \tau_{x+y}\, dx=\\
&=\frac 1{|B(0,R)|}\int_{B(0,R)}\tau_{-x} T\tau_{x+y}\, dx + G_R,
\end{aligned}
\end{equation}
where $\|G_R\|=O(1/R)$.

Thus $\|T_{R_n}\tau -\tau T_{R_n}\|\to 0$ for any translate $\tau$, but
since $T_{R_n}\tau-\tau T_{R_n}\to \widetilde{T}\tau -\tau \widetilde{T}$ in
the weak operator topology, then $\widetilde{T}\tau=\tau \widetilde{T}$.

Since $\widetilde{T}$ is a bounded linear operator that commutes with the
translations then it is a convolution operator, i.e. it is given by a Fourier
multiplier against a bounded function.

The fact that it is a projection onto the functions with spectra lying in $K$
implies that $\widetilde T$ is the  multiplier given by $\chi_K$. 
\end{proof}

\begin{corollaire}
For any smooth compact $K\subset\R^n$ there are no complete interpolating
sequences for the space $E_K^p$ for any $p\ne 2$, $1<p<\infty$.
\end{corollaire}

\begin{proof}
If such $\Lambda$ existed, $\chi_K$ will be a multiplier in $\mathcal F(L^p)$
but Fefferman theorem states that it cannot be bounded  because $K$ has points
with positive curvature. 
\end{proof}

\end{document}